\documentclass[11pt,a4paper,reqno]{amsart}

\usepackage{ulem}
 
 \usepackage{cite}
\usepackage{amsmath,amssymb,amsfonts,amsthm,mathtools}
\usepackage{enumerate}
\usepackage[colorlinks=true,citecolor=black,pagebackref=false]{hyperref}

\numberwithin{equation}{section}
\allowdisplaybreaks

\newcommand{\R}{\mathbb R}
\newcommand{\N}{\mathbb N}

\newcommand{\cZ}{\mathcal Z}

\makeatletter
\@namedef{subjclassname@2020}{\textup{2020} Mathematics Subject Classification}
\makeatother
\newtheorem{thm}{Theorem}[section]
\newtheorem{prop}[thm]{Proposition}
\newtheorem{lem}[thm]{Lemma}

\newtheorem{defi}[thm]{Definition}
\newtheorem{rem}[thm]{Remark}

\newtheorem{ques}[thm]{Question}

\title[]{On the spectrality of the non-homogeneous golden-mean self-similar measure}

\author{Yi-Qiu Mao}
\address{Y.-Q. Mao, School of Mathematics and Information Science, Guangzhou University, Guangzhou, 510006, P.~R.~China}
\email{yqmao@gzhu.edu.cn}

\author{Zhi-Yi Wu}
\address{Z.-Y. Wu, School of Mathematics and Information Science, Guangzhou University, Guangzhou, 510006, P.~R.~China}

\email{zywu@gzhu.edu.cn}

\subjclass[2020]{Primary 42B10; Secondary 28A80}

\keywords{golden-mean self-similar measure, spectrality, zero set}

\date{}

\begin{document}
\begin{abstract}
We investigate the spectral properties of a class of inhomogeneous self-similar measures, which does not admit a non-trivial infinite convolution structure. A central example is the golden-mean self-similar measure $\mu$, for which the existence of an exponential orthonormal basis in the associated $L^2$-space has remained a long-standing open problem. The usual approach for homogeneous self-similar measures  does not apply here, new methods are required. We establish several basic properties of the measure and then carry out a detailed numerical study of the zero set of its Fourier transform. Using a scanning and refinement algorithm that combines uniform grid sampling, quadratic interpolation, and golden-section search, we examine a wide range and find no real zeros of $\widehat{\mu}$, which provides concrete evidence that $\mu$ is very likely non-spectral, suggesting that inhomogeneity may serve as a natural obstruction to the existence of exponential orthonormal bases. To the best of our knowledge, our paper is the first attempt to study the spectrality of such measures through a combined analytic and numerical framework.
\end{abstract}

\maketitle
\section{Introduction}

A compactly supported Borel probability measure $\nu$ on $\mathbb{R}^d$ is called {\it spectral} if there exists a countable set $\Lambda\subseteq\mathbb{R}^d$ such that
$$
E(\Lambda)=\{e^{2\pi i\langle\lambda, x\rangle}:\lambda\in\Lambda\}
$$
forms an orthonormal basis for $L^{2}(\nu)$. The set $\Lambda$ is then called a {\it spectrum} of $\nu$.

Research on spectral measures originated with Fuglede \cite{Fug74}, who formulated his celebrated conjecture: the normalized Lebesgue measure on a set $\Omega\subseteq\mathbb{R}^d$ of positive finite Lebesgue measure is a spectral measure (also call $\Omega$ a {\it spectral set}) if and only if $\Omega$ tiles $\mathbb{R}^d$ by translations. Both directions of the conjecture are false for $d\ge 2$; see \cite{Z26} and the references therein. For $d=1$, however, both directions remain open. This problem has motivated extensive research into the properties that make a set or a measure spectral.

In 1998, Jorgensen and Pedersen \cite{JP98} constructed the first example of a non-atomic singular spectral measure, i.e., the standard middle-fourth Cantor measure $\mu_{4,\{0,2\}}$. Following this discovery, the spectrality of fractal measures including self-similar measures, self-affine measures, and Moran measures has been extensively explored; see \cite{AFL19,DC21,DHLai19,DL14,LW02} and the references therein. A crucial common feature of all these known spectral measures is that they can be expressed as an infinite convolution of finitely supported discrete measures. This convolution structure is what makes it possible to construct explicit spectra, and it remains the foundation of the known sufficient conditions for spectrality.

However, not all fractal measures admit such an infinite convolution structure. In this paper, we consider non-homogeneous self-similar measures. Let $\rho_1, \rho_2 \in (0, 1)$ and $p \in (0, 1)$. Let $f_{1, \rho_1}(x) = \rho_1 x$ and $f_{2, \rho_2}(x) = \rho_2(x + 2)$ be an iterated function system on $\mathbb{R}$. By Hutchinson's theorem \cite{Hut81}, there exists a unique Borel probability measure $\mu=\mu_{\rho_1, \rho_2}^p$ with non-empty compact support $K=K_{\rho_1,\rho_2}$ satisfying that
\begin{align}\label{def:self-s}
\mu(E) = p\,\mu\bigl(f_{1, \rho_1}^{-1}(E)\bigr) + (1 - p)\,\mu\bigl(f_{2, \rho_2}^{-1}(E)\bigr),
\end{align}
for any Borel set $E\subseteq \mathbb{R}$ and $K$ is the unique compact set satisfying
\begin{align}\label{defadd:self-sset}
K=f_{1,\rho_1}(K)\cup f_{2,\rho_2}(K).
\end{align}
The probability measure $\mu$ satisfying \eqref{def:self-s} is called a {\it self-similar measure} and $K$ is called a {\it self-similar set} or {\it attractor}. When $\rho_1 = \rho_2 =:\rho$, the measure $\mu=\mu_{\rho}^p := \mu_{\rho, \rho}^p$ is the classical Bernoulli convolution. In particular, taking $\rho = 1/4$ and $p=1/2$, $\mu_{\rho}^p$ is exactly the aforementioned $\mu_{4,\{0,2\}}$.

For the classical Bernoulli convolution, it is known that $\mu_{\rho}^p$ is a spectral measure if and only if $\rho =1/(2q)$ for some $q \in \N$ and $p = \frac{1}{2}$ \cite{Dai12,DHLai19}. For the case $\rho_1 \neq \rho_2$, its spectrality remains largely mysterious. The reason why the non-homogeneous case $\rho_1 \neq \rho_2$ is much more difficult than the homogeneous one is that while $\mu_\rho^p$ is an infinite convolution of Bernoulli measures, $\mu_{\rho_1, \rho_2}^p$ is not. The convolution structure is crucial to all the previous results. A long-standing folklore open problem in the community of fractal spectral measure theory is the following:

\begin{ques}
For $\rho_1\neq \rho_2$, when is $\mu_{\rho_1, \rho_2}^p$ a spectral measure?
\end{ques}

We write
\begin{equation}\label{eq:ssim}
s_{\lambda_1,\lambda_2}^p = \frac{p\log p + (1-p)\log(1-p)}{p\log\lambda_1 + (1-p)\log\lambda_2} \tag{1.2}
\end{equation}
for the similarity dimension of the measure $\mu_{\rho_1,\rho_2}^p$. To the best of our knowledge, apart from the classical setting, there are virtually no methods or results concerning the spectrality of $\mu_{\rho_1,\rho_2}^p$ in the absolutely continuous case. In this paper, we focus primarily on the singular case (a well-known sufficient condition for the measure to be singular is $s_{\rho_1,\rho_2}^p < 1$). However, even in this singular case, the problem remains notoriously difficult. Therefore, we restrict our attention to the condition $\rho_1+\rho_2 < 1$ (which is stronger than $s_{\rho_1,\rho_2}^p < 1$), under which the singularity of the measure is guaranteed and the analysis becomes relatively tractable.

In the case \(\rho_1 + \rho_2 < 1\), He, Kang, Tang and the second author \cite{HKTW18} obtained a necessary condition for the spectrality of \(\mu_{\rho_1, \rho_2}^p\). Specifically, they proved that if there exists a spectrum \(\Lambda\) of \(\mu_{\rho_1, \rho_2}^p\) such that its Beurling dimension is equal to the Hausdorff dimension of $K_{\rho_1,\rho_2}$ (a natural condition that holds in the classical Bernoulli convolution case), then \(p\) must satisfy \(p = \rho_1^s\),
where \(s\) is determined by \(\rho_1^s + \rho_2^s = 1\). In this setting, $s$ coincides with $s_{\rho_1,\rho_2}^p$ defined in \eqref{eq:ssim}.

On the other hand, up to now, every known singular continuous spectral measure has no Fourier decay. For $\mu_{\rho_1, \rho_2}^p$, it has Fourier decay whenever the two contraction ratios are not logarithmically commensurable, i.e., $\log_{\rho_2}\rho_1\notin \mathbb{Q}$ \cite{LS22}. Motivated by this fact, we shall restrict our attention to the logarithmically commensurable case.

In this paper, we focus on the simplest case of the general framework described above, namely the {\it golden-mean self-similar measure}, which we denote by $\mu$. It is defined by choosing the contraction ratios \(\rho_1 = \frac12, \rho_2 = \frac14\) and \(p = \left(\frac12\right)^s\),
where $s$ is determined by
\[
\left(\frac12\right)^s + \left(\frac14\right)^s = 1.
\]
Equivalently, $s = \log_2 \varphi$, where $\varphi = \frac{1+\sqrt{5}}{2}$ is the golden ratio; this gives rise to the name ``golden-mean'' self-similar measure.

There is no general result on Fourier decay for non-homogeneous self-similar measures with logarithmically commensurable contraction ratios. However, for the golden-mean measure considered in this paper, we can prove that it has no Fourier decay (see Proposition \ref{prop:non-convol} in Section 3).

It is well known that if $\mu$ is a spectral measure with spectrum $\Lambda$, then the orthogonality of $E(\Lambda)$ implies that
$
\widehat{\mu}(\lambda-\lambda')=0
$
for all distinct $\lambda,\lambda'\in\Lambda$, where
\[\widehat{\mu}(\xi)=\int e^{-2\pi i \xi x}\,\mathrm{d}\mu(x)\]
is the Fourier transform of $\mu$.  Since $\mu$ is non-atomic, $L^2(\mu)$ is infinite-dimensional; consequently, any spectrum $\Lambda$ must be countably infinite, and $\widehat{\mu}$ must vanish on the infinite difference set $(\Lambda-\Lambda)\setminus\{0\}$. Thus, understanding the zero set of $\widehat{\mu}$ is a natural first step toward determining the spectrality of $\mu$. In particular, if $\widehat{\mu}$ has no real zeros, or has at most finitely many real zeros, then $\mu$ cannot admit an infinite orthogonal system of exponentials, and hence cannot be spectral.

In this paper, we first set up some basic properties of the gold-mean self-similar measure $\mu$. Then we develop a scanning and refinement algorithm that combines uniform grid sampling, complex quadratic interpolation, and golden-section search to locate potential zeros of the auxiliary Riccati function $q(\xi)=\widehat{\mu}(\xi)/\widehat{\mu}(\xi/2)$ with high precision. The algorithm proceeds in two phases: a coarse scan over a uniform grid to detect local minima of $|q|$, followed by a refined search using quadratic interpolation and golden-section search to pinpoint candidate zeros. Using this algorithm, we examine the interval $[-10^7,10^7]$ and find no real zeros of $\widehat{\mu}$. Together with the theoretical analysis on the interval $[-1.08,1.08]$ (see Proposition~\ref{prop-nozero}), this provides concrete numerical evidence that $\mu$ is very likely non-spectral.

Although the interval $[-1.08,1.08]$ derived by the rigorous analysis may seem small, this is not an essential limitation. A larger interval could be obtained by more refined estimates, but the required technical effort would be considerable and, in our view, disproportionate to the additional benefit, especially since the numerical scan already rules out zeros with high precision over the much larger interval $[-10^7,10^7]$.

\section{Elementary properties of the golden-mean self-similar measure}
In this section, we will set up some basic properties of the gold-mean self-similar measure $\mu$. Some of these may be known in more general case, but we give simpler proofs here. For simplicity, we write $f_1(x)=f_{1,1/2}(x)=x/2$ and $f_2(x)=f_{2,1/4}(x)=(x+2)/4$. Let $\Sigma_2=\{0,1\}$. For $n\ge1$, let $\Sigma_2^n=\{I=i_1i_2\cdots i_n: i_k\in\Sigma_2~\text{for}~1\leq k\leq n\}$ be the set of words of length $n$ and let $\Sigma_2^{\infty}=\{I=i_1i_2\cdots: i_k\in\Sigma_2~\text{for}~k\leq 1\}$ be the set of infinite words.

\subsection{Self-similar set and binary expansions}

We now give a convenient description of the self-similar set $K$ associated to the gold-mean self-similar measure in terms of binary expansions.
\begin{prop}\label{prop:shows}
Let $K$ be the attractor of the golden-mean iterated function system $\{f_1,f_2\}$. Then
\[
K=
\left\{
\sum_{n=1}^{\infty}\frac{\varepsilon_n}{2^n}:
\varepsilon_n\in\Sigma_2
\text{ and }
\varepsilon_n\varepsilon_{n+1}=0
\text{ for every }n\geq 1
\right\}.
\]
Equivalently, \(K\) consists of the points in \([0,1]\) that admit a binary expansion with no two consecutive \(1\)'s.
\end{prop}
\begin{proof}
Let
\[
A=\{\{\varepsilon_n\}_{n\ge1}\in\Sigma_2^{\infty}:\varepsilon_n\varepsilon_{n+1}=0\text{ for all }n\ge1\}.
\]
Equip $\Sigma_2^{\infty}$ with the standard product topology and it is well-known that  $\Sigma_2^{\infty}$ is compact. It is easy to see that the set $A$ is closed and thus $A$ is compact. Note that the map
\[
\Phi(\{\varepsilon_n\})=\sum_{n=1}^\infty \frac{\varepsilon_n}{2^n},\quad \{\varepsilon_n\}_{n\ge1}\in\Sigma_2^{\infty}
\]
is continuous and so $S=\Phi(A)$ is compact. Moreover, $S$ is nonempty because the zero sequence belongs to $A$. Then, by \eqref{defadd:self-sset} we only need to show that $S$ satisfies
\[
S = f_1(S)\cup f_2(S).
\]

Take $x\in S$ and write $x = \sum_{n=1}^\infty a_n/2^n$ with all $a_n\in\{0,1\}$. If $a_1 = 0$, then
\[
x = \sum_{n=2}^{\infty} \frac{a_n}{2^n}
    = \frac12 \sum_{n=1}^{\infty} \frac{a_{n+1}}{2^n}=:\frac{1}{2}y.
\]
Since $\{a_{n+1}\}_{n=1}^\infty$ contains no consecutive $1$'s, it follows that $y\in S$ and $x = f_1(y) \in f_1(S)$. If $a_1 = 1$, then the definition of $S$ forces $a_2 = 0$. Hence
\[
x = \frac12 + \frac{0}{4} + \sum_{n=3}^{\infty} \frac{a_n}{2^n}
    = \frac12 + \frac14 \sum_{n=1}^{\infty} \frac{a_{n+2}}{2^n}=:\frac12 + \frac14z.
\]
The sequence $(a_{n+2})_{n=1}^\infty$  again contains no consecutive $1$'s.Then $z \in S$ and $x = f_2(z) \in f_2(S)$. In either case $x$ lies in $f_1(S)\cup f_2(S)$, proving $S\subseteq f_1(S)\cup f_2(S)$.

To prove the other direction, if $x\in f_1(S)$, then $x = y/2$ for some $y\in S$. Write $y = \sum_{n=1}^\infty b_n/2^n$ with all $b_n\in\{0,1\}$. Then
\[
x = \sum_{n=1}^{\infty} \frac{b_n}{2^{n+1}}
    = \sum_{n=1}^{\infty} \frac{a_n}{2^n},
\]
where $a_1 = 0$ and $a_{n+1} = b_n$ for $n\ge 1$. The sequence $\{a_n\}$ contains no consecutive $1$'s from $\{b_n\}$ and then $x\in S$. If $x\in f_2(S)$. Then $x = \frac12 + z/4$ for some $z\in S$. Write $z = \sum_{n=1}^\infty b_n/2^n$ with all $b_n\in\{0,1\}$. Then
\[
x = \frac12 + \sum_{n=1}^{\infty} \frac{b_n}{2^{n+2}}
    = \frac12+ \sum_{n=3}^{\infty} \frac{b_{n-2}}{2^n}\in S.
\]
Thus $f_1(S)\cup f_2(S)\subseteq S$.
\end{proof}

\subsection{Singular continuity of \(\mu\)}

\begin{prop}\label{prop:mu-singular}
The golden-mean self-similar measure \(\mu\) is non-atomic and singular with respect to Lebesgue measure.
\end{prop}

\begin{proof}
We first prove that $\mu$ is non-atomic. Let \(p_1=2^{-s}\) and \(p_2=4^{-s}\), where \(s>0\) is defined by \(p_1+p_2=1\). Then \(0<p_1,p_2<1\).

For a word \(\mathbf i=i_1\cdots i_n\) with $i_j\in \{1,2\}$ for $1\leq j\leq n$, write \(f_{\mathbf i}=f_{i_1}\circ\cdots\circ f_{i_n}\) and \(K_{\mathbf i}=f_{\mathbf i}(K)\). By Proposition \ref{prop:shows}, for any \(x=\sum_{n\ge1}\varepsilon_n2^{-n}\in K\) with all $\varepsilon_n\in\{0,1\}$, the no consecutive ones constraint gives
\[
x\le \sum_{k=0}^{\infty}2^{-(2k+1)}
=\frac23,
\]
i.e., $x\in\big[0,\frac23\big]$. Hence \(f_1(K)=\frac12K\subseteq[0,\frac13]\), and
\[
f_2(K)=\frac14(K+2)\subseteq \frac14\left[0,\frac23\right]+\frac12
=\left[\frac12,\frac23\right],
\]
which implies that \(f_1(K)\cap f_2(K)=\varnothing\). Therefore, by \eqref{def:self-s}
\[
\mu(K_{\mathbf i})=p_{i_1}\cdots p_{i_n}\le p_{\max}^n,
\]
where $p_{\max}:=\max\{p_1,p_2\}<1$.

For any \(x\in K\), there exists an \(\omega=\omega_1\omega_2\cdots\) with all $\omega_i\in\{1,2\}$ such that \(x\in K_{\omega_1\cdots\omega_n}\) for every \(n\). Therefore
\[
\mu(\{x\})\le \mu(K_{\omega_1\cdots\omega_n})\le p_{\max}^n\to 0,
\]
so \(\mu(\{x\})=0\). Thus \(\mu\) has no atoms.

It remains to show \(\mu\perp\mathcal{L}\), where $\mathcal{L}$ denotes the Lebesgue measure. Let \(I=[0,2/3]\). For $n\geq1$ and each \(\mathbf i=i_1i_2\cdots i_n\) with all $i_j\in\{1,2\}$, denote \(I_{\mathbf i}=f_{\mathbf i}(I)\). Its length is $|I_{\mathbf i}|=|I|\prod_{j=1}^n \rho_{i_j}$, where $\rho_1=1/2$ and $\rho_2=1/4$.

Since \(K\subseteq\bigcup_{|\mathbf i|=n}K_{\mathbf i}\subseteq \bigcup_{|\mathbf i|=n}I_{\mathbf i}\), it follows that
\[
\mathcal{L}(K)\le \sum_{|\mathbf i|=n} |I_{\mathbf i}|
=\frac23 \sum_{|\mathbf i|=n}\prod_{j=1}^n \rho_{i_j}
=\frac23\left(\frac12+\frac14\right)^n
=\frac23\left(\frac34\right)^n,
\]
where $|\mathbf i|$ denotes the length of $\mathbf i$. Thus \(\mathcal{L}(K)=0\) by letting $n\to\infty$. Since \(\mu\) is supported on \(K\), it follows that \(\mu\perp\mathcal{L}\). The proof is complete.
\end{proof}

\subsection{Not an infinite Bernoulli convolution}

\begin{thm}\label{thm-nonconvolution}
The golden-mean self-similar measure \(\mu\) cannot be represented as an infinite convolution of Bernoulli measures.
\end{thm}

To prove it, we need some preliminaries. We first recall the definition of an infinite convolution of Bernoulli measures.
\begin{defi}
\rm Let $\{r_k\}_{k\ge 1}$ be a sequence of positive real numbers satisfying $\sum_{k=1}^{\infty} r_k < \infty$, and let $\{p_k\}_{k\ge 1}$ be a sequence with $p_k\in(0,1)$ for all $k\ge 1$. We say that
\begin{align}\label{def:bernoullic}
\nu=\mathop{*}_{k=1}^{\infty}
\bigl(p_k\delta_0+(1-p_k)\delta_{r_k}\bigr)
\end{align}
is an {\it infinite convolution of Bernoulli measures}, where the convergence is in weak sense.

It is easy to see that the support of $\nu=\mathop{*}_{k=1}^{\infty}
\bigl(p_k\delta_0+(1-p_k)\delta_{r_k}\bigr)$ is the following set
\[
\mathcal A(r_k)
=
\left\{
\sum_{k=1}^{\infty}\varepsilon_k r_k:
\varepsilon_k\in\{0,1\}
\right\}.
\]
\end{defi}

\begin{prop}\label{prop:ssymefea}
Let $\nu$ be an infinite convolution of Bernoulli measures defined in \ref{def:bernoullic}. Then \(\mathcal A(r_k)\) is centrally symmetric with respect to \(\frac12\sum_{k=1}^{\infty}r_k\), i.e., if
$x\in\mathcal A(r_k)$, then  $\sum_{k=1}^{\infty}r_k-x\in\mathcal A(r_k)$.
\end{prop}
\begin{proof}
Let \(x=\sum_{k=1}^{\infty}\varepsilon_k r_k
\in\mathcal A(r_k)\) with
\(\varepsilon_k\in\{0,1\}\). Then
\[
\sum_{k=1}^{\infty}r_k-x
=
\sum_{k=1}^{\infty}(1-\varepsilon_k)r_k\in \mathcal A(r_k),
\]
since \(1-\varepsilon_k\in\{0,1\}\).
\end{proof}

\begin{lem}\label{lem:daewfawfeew}
The attractor \(K\) of the golden-mean iterated function system is not centrally symmetric.
\end{lem}

\begin{proof}
We first note that \(\min K=0\) and \(\max K=\frac23\) by Proposition \ref{prop:shows}. If \(K\) is centrally symmetric, then its center would necessarily be
the midpoint of its smallest and largest points. Thus the only possible
center would be
\[
c=\frac{0+2/3}{2}=\frac13.
\]

Note that
\(\frac14=\sum_{j=1}^{\infty}\frac{\varepsilon_j}{2^j}\) with
$\{\varepsilon_j\}_{j\geq1}=\{0,1,0,0,\ldots\}$. Then \(\frac14\in K\). Consider $2c-\frac14=\frac5{12}$. It is easy to see that
\[
\frac5{12}=0.011010101\cdots{}_2=0.01\overline{10}_2
\]
whose binary expansion contains two consecutive \(1\)'s. Hence \(\frac5{12}\notin K\).

Thus \(K\) is not symmetric with respect to \(1/3\), and consequently it is
not centrally symmetric.
\end{proof}

\begin{proof}[Proof of Theorem~\ref{thm-nonconvolution}]
Suppose, to the contrary, that,
\[
\mu=
\mathop{*}_{k=1}^{\infty}
\bigl(p_k\delta_0+(1-p_k)\delta_{r_k}\bigr),
\]
where \(
r_k>0\), \(p_k\in(0,1)\) and \(\sum_{k=1}^{\infty}r_k<\infty\).

The support of the convolution on the right-hand side is
\[
\mathcal A(r_k)
=
\left\{
\sum_{k=1}^{\infty}\varepsilon_k r_k:
\varepsilon_k\in\{0,1\}
\right\}.
\]
By Proposition \ref{prop:ssymefea}, \(\mathcal A(r_k)\) is centrally symmetric.
Consequently, any translate of \(\mathcal A(r_k)\) is also centrally
symmetric.

On the other hand, the support of the golden-mean self-similar measure is
the attractor \(K\), and Lemma \ref{lem:daewfawfeew} shows that \(K\) is not
centrally symmetric. This is a contradiction.

Hence, we complete the proof.
\end{proof}

\begin{rem}
Theorem \ref{thm-nonconvolution} merely demonstrates that $\mu$ cannot arise as an infinite convolution of two-point probability measures. In fact, Professor Guo-Tai Deng \cite{Deng26} pointed to us that the $\mu$ cannot be expressed as a non-trivial infinite convolution of finite discrete measures, which needs a more delicate analysis.
\end{rem}

\section{The Fourier transform and Riccati equation}
Write \(F(\xi)=\widehat{\mu}(\xi)\) and set
\(p=\bigl(\tfrac12\bigr)^s=\frac{\sqrt{5}-1}{2}\).
From \eqref{def:self-s}, we obtain for any $\xi\in \mathbb{R}$,
\begin{equation}\label{eq:3.1}
F(\xi)=pF(\xi/2)+p^2 e^{-\pi i\xi}F(\xi/4).
\end{equation}
By the definition of the Fourier transform of \(\mu\), $F(0)=1$.

Define auxiliary measures $\mu_A(\cdot)=\mu(\cdot)$, $\mu_B(\cdot)=\mu(f_1^{-1}(\cdot))$
and their Fourier transforms $F_A,F_B$. Since \(f_1(x)=x/2\), we have
\begin{equation}\label{eq:addfoure}
\begin{aligned}
F_B(\xi)&= \widehat{\mu_B}(\xi)= \int_{\mathbb{R}} e^{-2\pi i \xi x}\,\mathrm{d}\mu_B(x)= \int_{\mathbb{R}} e^{-2\pi i \xi f_1(y)}\,\mathrm{d}\mu(y)\\
&= \int_{\mathbb{R}} e^{-2\pi i \xi y/2}\,\mathrm{d}\mu(y)= F_A(\xi/2).
\end{aligned}
\end{equation}

Using \eqref{eq:3.1} and \eqref{eq:addfoure}, we compute
\[
F_A(\xi)
= pF_A(\xi/2)+p^2 e^{-\pi i\xi}F_A(\xi/4)
= pF_A(\xi/2)+p^2 e^{-\pi i\xi}F_B(\xi/2).
\]
Together with \eqref{eq:addfoure} again, this yields the matrix recursion
\begin{equation}\label{eq:matrix-recursion}
 \begin{pmatrix}F_A(\xi)\\F_B(\xi)\end{pmatrix}
 =
 M(\xi)
 \begin{pmatrix}F_A(\xi/2)\\F_B(\xi/2)\end{pmatrix},
\end{equation}
where
\[
 M(\xi)=
 \begin{pmatrix}
  p & p^{2}e^{-\pi i\xi}\\
  1 & 0
 \end{pmatrix}.
\]
Iterating \eqref{eq:matrix-recursion} \(N\) times gives
\[
\begin{pmatrix}F(\xi)\\F(\xi/2)\end{pmatrix}
=
\Biggl(\prod_{k=0}^{N-1}M(\xi/2^{k})\Biggr)
\begin{pmatrix}F_A(\xi/2^N)\\F_B(\xi/2^N)\end{pmatrix}.
\]
By the continuity of the Fourier transform of a measure, taking $N\to\infty$ yields the convergent representation
\begin{equation}\label{eq:infinite-product}
 \begin{pmatrix}F(\xi)\\F(\xi/2)\end{pmatrix}
 =
 \lim_{N\to\infty}
 \Biggl(\prod_{k=0}^{N-1}M(\xi/2^{k})\Biggr)
 \begin{pmatrix}1\\1\end{pmatrix}.
\end{equation}

\begin{prop}\label{prop:zero-rigidity}
For every $\xi\in\R$, $(F_A(\xi),F_B(\xi))\neq(0,0)$.
Consequently, if \(F(\xi)=0\), then \(F(\xi/2)\neq0\). In particular, every real zero \(\xi\) of \(F\) satisfies
\begin{equation}\label{eq:zero-phase}
 \frac{F(\xi/2)}{F(\xi/4)}=-p\,e^{-\pi i\xi}.
\end{equation}
\end{prop}
\begin{proof}
Assume, for a contradiction, that there exists $\xi_0\in\R$ with
\[(F_A(\xi_0),F_B(\xi_0))=(0,0),\]
i.e.,\ $F(\xi_0)=F(\xi_0/2)=0$. Since $\det M(\xi)=-p^{2}e^{-\pi i\xi}\neq0$ for all $\xi\in\mathbb{R}$, it follows that $M(\xi)$ is invertible. Applying $M(\xi_0)^{-1}$ to  the left in \eqref{eq:matrix-recursion} yields
\[(F(\xi_0/2),F(\xi_0/4))=(0,0).\]
Thus, by induction,
\[
F(\xi_0/2^{n})=0\qquad\text{for all } n\ge0.
\]

Since $\mu$ is a probability measure, $F$ is
uniformly continuous on $\R$.  Letting $n\to\infty$ we obtain
$F(0)=\lim_{n\to\infty}F(\xi_0/2^{n})=0$, contradicting $F(0)=1$.
Hence $(F_A(\xi),F_B(\xi))\neq(0,0)$ for every $\xi\in\R$.

Now suppose $F(\xi)=0$.  The non-vanishing of the pair
$(F(\xi),F(\xi/2))$ forces $F(\xi/2)\neq0$.  Insert $F(\xi)=0$ into \eqref{eq:3.1} to obtain
\begin{equation}\label{eq:addzerefea}
pF(\xi/2)+p^{2}e^{-\pi i\xi}F(\xi/4)=0.
\end{equation} If $F(\xi/4)$
is zero, this would imply $F(\xi/2)=0$, contradicting
$F(\xi/2)\neq0$; therefore $F(\xi/4)\neq0$.  Dividing \eqref{eq:addzerefea} by
$pF(\xi/4)$ gives the desired result
$\displaystyle\frac{F(\xi/2)}{F(\xi/4)}=-p\,e^{-\pi i\xi}$.
\end{proof}

The following proposition is well-known and easy to prove.
\begin{prop}\label{prop:difference-set}
If $E(\Lambda)$ is orthogonal in $L^{2}(\mu)$, then
$(\Lambda-\Lambda)\setminus\{0\}\subseteq\cZ(F)$. Thus $\cZ(F)=\varnothing$ implies $\#(\Lambda)\leq1$ and
$\mu$ is non-spectral.
\end{prop}

\begin{prop}\label{prop-nozero}
\(F\) does not have a zero in the interval \([-1.08,1.08]\).
\end{prop}

\begin{proof}
Since \(\mu\) is a positive measure, we have \(F(-\xi)=\overline{F(\xi)}\). Hence it suffices to prove that \(F\) has no zero in \([0,1.08]\). Define
\[
m=\sup\{\xi\ge 0: F(t)\ne 0 \text{ for all }t\in[0,\xi]\}.
\]
We will show \(m>1.08\).

First, for \(\xi\in(0,3/4]\), since \(\operatorname{supp}\mu\subseteq[0,2/3]\), we have \(2\pi\xi x\in[0,\pi]\) for \(\mu\)-a.e. \(x\). Therefore
\[
-\operatorname{Im}F(\xi)
=\int_{\mathbb R}\sin(2\pi\xi x)\,\mathrm{d}\mu(x)
>0.
\]
Thus \(F(\xi)\ne 0\), and so \(m>3/4\).

Now let \(\xi\in[3/4,1)\). By \eqref{def:self-s}, the intervals \(A=[0,1/6]\) and \(B=[1/2,2/3]\) both have \(\mu\)-measure \(p^2\). Moreover, for \(\mu\)-a.e. \(x\in K\setminus(A\cup B)\), we have \(x\in(1/6,1/2)\), hence
\[
2\pi\xi x\in(\pi\xi/3,\pi\xi)\subseteq(\pi/4,\pi),
\]
and therefore \(\sin(2\pi\xi x)\ge 0\). Consequently,
\begin{align}\label{eq:adddadeed}
\int_{\mathbb R}\sin(2\pi\xi x)\,\mathrm{d}\mu(x)
&\ge \int_A \sin(2\pi\xi x)\,\mathrm{d}\mu(x)+\int_B \sin(2\pi\xi x)\,\mathrm{d}\mu(x) \nonumber \\
&=\int_A \left(\sin(2\pi\xi x)+\sin(2\pi\xi(x+1/2))\right)\,\mathrm{d}\mu(x).
\end{align}

Using \(\sin u+\sin v=2\sin\frac{u+v}{2}\cos\frac{u-v}{2}\), the integrand in \eqref{eq:adddadeed} equals
\[
2\sin\left(2\pi\xi x+\frac{\pi\xi}{2}\right)\cos\left(\frac{\pi\xi}{2}\right).
\]
For \(x\in A\), we have
\[
2\pi\xi x+\frac{\pi\xi}{2}\in\left[\frac{\pi\xi}{2},\frac{5\pi\xi}{6}\right]
\subseteq(0,\pi),
\]
so the sine factor is positive, and also \(\cos(\pi\xi/2)>0\). Therefore the integrand in \eqref{eq:adddadeed} is strictly positive on \(A\), and thus
\[
\int_{\mathbb R}\sin(2\pi\xi x)\,\mathrm{d}\mu(x)>0.
\]
It follows that \(F(\xi)\ne 0\), and hence \(m> 1\).

At \(\xi=1\), we have the sharper lower bound
\[
\int_{\mathbb R}\sin(2\pi x)\,\mathrm{d}\mu(x)
\ge \int_{[1/4,1/3]}\sin(2\pi x)\,\mathrm{d}\mu(x)
> \sin\left(\frac{2\pi}{3}\right)p^3>0.
\]

Now define
\[
G(\xi)=\int_{\mathbb R}\sin(2\pi\xi x)\,\mathrm{d}\mu(x).
\]
We have just shown that
\[
G(1)>\frac{\sqrt{3}}{2}p^3=:H.
\]
Moreover, by the dominated convergence theorem,
\begin{align*}
\left|\frac{\mathrm{d}}{\mathrm{d}\xi}G(\xi)\right|
&\le 2\pi \int_0^{2/3}x\,\mathrm{d}\mu(x)\le 2\pi\bigg(\frac{1}{6}p^2+\frac{p^3}{3}+\frac{2}{3}p^2\bigg)\\
&=2\pi\left(\frac{1}{3}+\frac{p^2}{6}\right)=:L,
\end{align*}
where the last equality uses \(p^2+p=1\).

Then, for any \(\xi>1\), the mean value theorem gives some \(\theta\in(1,\xi)\) such that
\[
G(\xi)-G(1)=G'(\theta)(\xi-1).
\]
Thus
\[
G(\xi)>H-L(\xi-1).
\]
Taking \(\xi=1+\frac{H}{L}\), we obtain \(G(\xi)>0\), so \(F(\xi)\ne 0\). Consequently,
\[
m\ge 1+\frac{H}{L}
=1+\frac{\frac{\sqrt{3}}{2}p^3}{2\pi\left(\frac{1}{3}+\frac{p^2}{6}\right)}>1.08.
\]

Hence, we complete the proof.
\end{proof}

Now define the {\it Riccati function}
\begin{equation*}
 q(\xi)=\frac{F(\xi)}{F(\xi/2)},
\end{equation*}
which is well-defined whenever \(F(\xi/2)\neq0\). From \eqref{eq:3.1}, whenever \(q(\xi/2)\neq0\), we have
\begin{equation}\label{eq:Riccati}
 q(\xi)=p+\frac{p^{2}e^{-\pi i\xi}}{q(\xi/2)}
\end{equation}
with $q(0)=1$.
Iterating this relation formally yields the infinite-product representation
\begin{equation}\label{eq:infinite-product-q}
 F(\xi)
 =\prod_{k=0}^{\infty}q(\xi/2^{k})\,F(0)
 =\prod_{k=0}^{\infty}q(\xi/2^{k}),
\end{equation}
provided that all factors are well-defined and the product converges. If the product converges, then \(F(\xi)=0\) exactly when
\[
\prod_{k=0}^{\infty}\left|q(\xi/2^{k})\right|\to0.
\]

As an immediate consequence of \eqref{eq:zero-phase}, we state the following reformulation.

\begin{rem}
If \(F(\xi)=0\) for some \(\xi\in\mathbb R\), then \(q(\xi/2)\) is well-defined and satisfies
\begin{equation*}
 q(\xi/2)=-p\,e^{-\pi i\xi},
\end{equation*}
and consequently
\begin{equation*}
 |q(\xi/2)|=p,
 \quad\text{and}\quad
 \arg q(\xi/2)\equiv-\pi\xi \pmod{2\pi}.
\end{equation*}
Thus, a real zero at \(\xi\) forces \(q(\xi/2)\) to have modulus exactly \(p\) and argument exactly \(-\pi\xi \pmod{2\pi}\). This is a strong two-parameter restriction at each dyadic level. For \(F\) to have many real zeros, these conditions must hold coherently across all dyadic levels, which is highly restrictive.
\end{rem}

At the end of this section. We prove the following:
\begin{prop}\label{prop:non-convol}
The golden-mean self-similar measure $\mu$ has no Fourier decay.
More precisely,
\[
\lim_{k\to\infty}\widehat{\mu}(2^{k}) = L \neq 0,
\]
so that $\limsup_{|\xi|\to\infty}|\widehat{\mu}(\xi)|\ge |L|>0$.
\end{prop}
\begin{proof}
Recall that \(p+p^2=1\), and the Fourier transform \(F=\widehat{\mu}\) satisfies
\begin{equation}\label{eq:idenfours}
F(\xi)=pF(\xi/2)+p^2 e^{-\pi i\xi}F(\xi/4).
\end{equation}
Set $a_k=F(2^k)$. For \(k\ge 2\), since \(e^{-\pi i 2^k}=1\), we have
\[
a_k=p a_{k-1}+p^2 a_{k-2}.
\]
The characteristic equation of the above recurrence relation is
\[
x^2-px-p^2=0,
\]
whose roots are \(1\) and \(-p^2\). By Theorem 7.2.2 in \cite{Brualdi09}, there exist constants \(a,b\) such that
\[
a_k=a+b(-p^2)^k.
\]
Since \(p<1\), \(\lim_{k\to\infty} a_k=a\) exists.

Using the initial conditions \(a_0=F(1)\) and \(a_1=F(2)\), we solve for \(a\):
\[
a=\frac{a_1+p^2a_0}{1+p^2}
=\frac{F(2)+p^2F(1)}{1+p^2}.
\]
Thus \(a=0\) is equivalent to
\begin{equation}\label{eq:minuserua}
F(2)=-p^2F(1).
\end{equation}

We now rule out this possibility. By the proof of Proposition~\ref{prop-nozero}, for all $0<\xi\le1$,
we have
\[
-\operatorname{Im}F(\xi)
=\int_{\mathbb R}\sin(2\pi\xi x)\,\mathrm{d}\mu(x)
>0.
\]
In particular,
$\operatorname{Im}F(1)<0$ and $\operatorname{Im}F(1/2)<0$.

Applying \eqref{eq:idenfours} with \(\xi=2\), we get
\[
F(2)=pF(1)+p^2F(1/2).
\]
Taking imaginary parts and using \(p,p^2>0\), we obtain
\begin{equation}\label{eq:addeqdfaefe}
\operatorname{Im}F(2)
=p\operatorname{Im}F(1)+p^2\operatorname{Im}F(1/2)
<0.
\end{equation}
On the other hand, if \eqref{eq:minuserua} holds, then
\[
\operatorname{Im}F(2)
=-p^2\operatorname{Im}F(1)>0,
\]
contradicting \eqref{eq:addeqdfaefe}. Therefore \(a\ne0\), and consequently
\[
\lim_{k\to\infty}\widehat{\mu}(2^k)=a\ne0.
\]
Hence
\[
\limsup_{|\xi|\to\infty}|\widehat{\mu}(\xi)|\ge a>0,
\]
and so \(\mu\) has no Fourier decay.
\end{proof}
\section{Numerical validation of the zero-free property}

We describe the computational procedures used to probe the zeroes of the Fourier
transform $F(\xi)=\widehat\mu(\xi)$ of the golden mean self-similar
measure, and report the main numerical
observations.
\subsection*{Calculation of the Riccati function}

The Fourier transform \(F(\xi)=\widehat{\mu}(\xi)\) of the golden-mean self-similar measure is shown by numerical experiments oscillating wildly with a trend toward zero when $\xi$ is large, see Figure \ref{fig:Flog}~(though not strictly approaching zero as is shown in Proposition \ref{prop:non-convol}). A direct zero hunting of \(F(\xi)\) for large \(|\xi|\) is then problematic (due to the risk of shallow valleys of $|F|$ near zero). To find zeros of $F$ more accurately, it is therefore more reliable to work with the Riccati function
\begin{equation}\label{eq: q computation}
	q(\xi)=\frac{F(\xi)}{F(\xi/2)},
\end{equation}
which satisfies the continued-fraction recursion \eqref{eq:Riccati}. Crucially, this recursion automatically renormalizes the value at each scale, keeping \(|q(\xi)|\) mostly within a bounded, moderate range (in fact mostly within $[0,1]$ as is observed in our numerical scans).

\begin{figure}[h]
	\centering
	\includegraphics[width=0.8\textwidth]{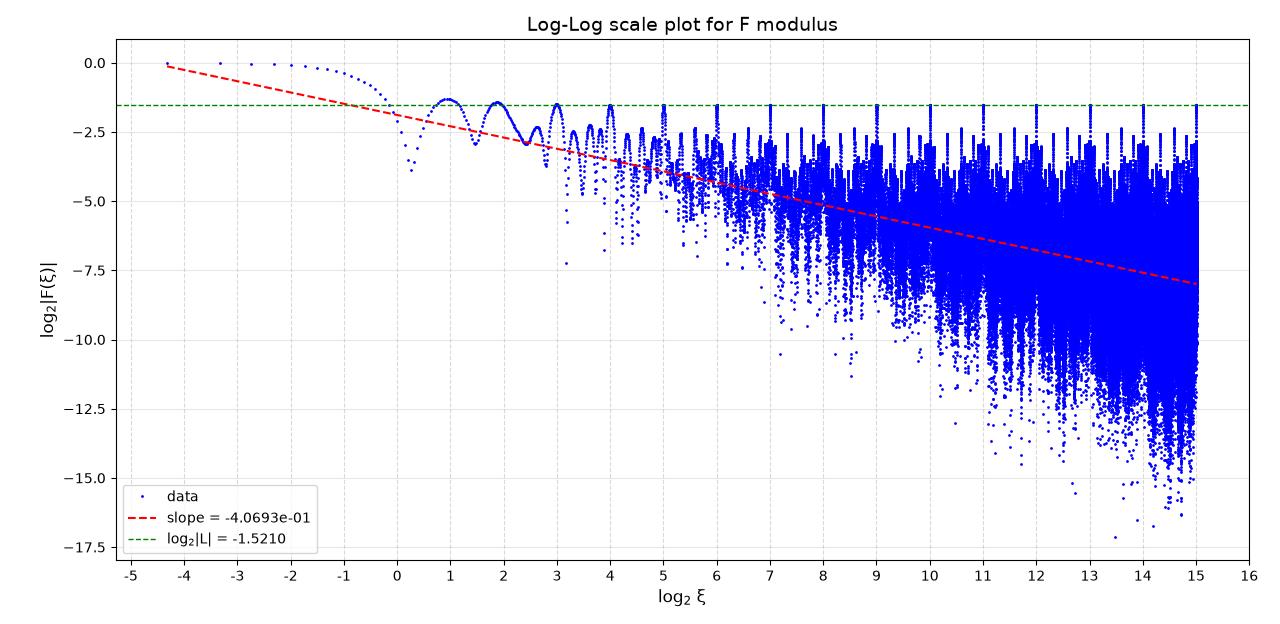}
	\caption{Plot of |F| with log scales. Notice the regular spikes corresponding to the limit number $L$ in Proposition \ref{prop:non-convol}.}
	\label{fig:Flog}
\end{figure}

The infinite product representation \(F(\xi)=\prod_{k=0}^{\infty}q(\xi/2^{k})\) (see \eqref{eq:infinite-product-q}), together with the nonzeroness of $F$ near $0$ implies that the smallest zeroes of $F$ and $q$ are equivalent. Thus, detecting zeros of \(q\) is sufficient for finding zeros for $F$.  The computation of $q$ is done by first truncating the infinite product to get an approximation to $F$ and then using \eqref{eq: q computation}.

\subsection{Algorithm for zero detection}

Our strategy for locating real zeros of \(q(\xi)\) consists of three stages:

\begin{enumerate}
	\item \textbf{Coarse uniform scan.}
	On any given interval \([a,b]\), we take a uniform step size \(h\) (typically \(h\leq 0.1\)) and compute \(q(x_i)\) and \(|q(x_i)|\) at each grid point \(x_i=a+ih\).
	The step size is chosen to be well below the oscillation period \(2\) of the exponential factor \(e^{-\pi i\xi}\), (the oscillation period \(2\) is also well observed from numerical observations below) thereby avoiding aliasing and ensuring that all local minima of \(|q|\) are related to local minimums of these sample points.
	
	\item \textbf{Complex quadratic interpolation and fast valley screening.}
	For every grid point \(x_i\) that satisfies
	\[
	|q(x_i)|\leq|q(x_{i-1})| \quad\text{and}\quad |q(x_i)|\leq|q(x_{i+1})|,
	\]
	we form the three-point complex quadratic interpolant \(P(\xi)\) through the points \((x_{i-1},q(x_{i-1})), (x_i,q(x_i)), (x_{i+1},q(x_{i+1}))\).
	Since \(q(\xi)\) is analytic on the real axis, the interpolation error is controlled by the third derivative of \(q\) and is of order \(O(h^3)\) in the interval \([x_{i-1},x_{i+1}]\).
	
	To decide whether this valley may contain a zero, we perform a {\it fast golden-section search} on \(|P(\xi)|\) with only $10$ iterations (plus an early exit if the estimated minimum falls below a given threshold \(\tau\)). This gives an estimated minimum value
	\[
	m_{\text{est}} = \min_{\xi\in[x_{i-1},x_{i+1}]} |P(\xi)|.
	\]
	If \(m_{\text{est}}<\tau\), the valley is flagged as a candidate for further refinement.
	
	\item \textbf{High-precision refinement on the original function.}
	For each flagged candidate, we apply a full golden-section minimization to the original function \(|q(\xi)|\) on the same interval \([x_{i-1},x_{i+1}]\), iterating until the interval width is below \(10^{-12}\). This yields the refined location \(\xi^*\) and the true minimum value \(|q(\xi^*)|\).
	If \(|q(\xi^*)|\) is below a very small threshold (e.g., \(10^{-10}\)), we consider it a strong numerical indication of a real zero; otherwise it is recorded as a shallow minimum.
	
\end{enumerate}

The algorithm is implemented in Python using double-precision floating-point arithmetic, with the infinite product evaluated at depth \(N=80\) (which gives an error below machine precision, at least for $F$, as shown in the error analysis below). The code is available upon request.

\subsection{Error analysis and reliability of zero detection}

We now justify that the above scanning procedure is capable of detecting every real zero in a given interval (we take $[0,10^7]$ in the following study) that may exist in the scanned interval, provided the step size \(h\) and threshold \(\tau\) are chosen appropriately. First we show the evaluation of $p$ at any given $\xi$ is accurate enough for our purpose.

\subsubsection{Error analysis for $q$ at given $\xi$}The numerical evaluation of the Riccati function $q(\xi)$ is based on the finite truncation of the infinite product \eqref{eq:infinite-product}. In the program we iterate $N=80$ levels and set the innermost vector to $(1,1)^T$. We choose $\xi\in [0,10^7]$. We now provide a error analysis, separating the truncation error (due to finite $N$) from the propagation of floating-point round-off errors.

{\it Truncation error}. Let $M(x)$ be the $2\times2$ matrix
\begin{equation*}
	M(x)=\begin{pmatrix}
		p & p^{2}e^{-\pi i x}\\
		1 & 0
	\end{pmatrix},
\end{equation*}
and define the product
\begin{equation*}
	\prod_{k=0}^{N-1} M(\xi/2^{k})
	=
	\begin{pmatrix}
		A_N & B_N\\
		C_N & D_N
	\end{pmatrix}.
\end{equation*}
The self-similarity of the Fourier transform implies the exact vector identity
\begin{equation*}
	\begin{pmatrix}
		F(\xi)\\
		F(\xi/2)
	\end{pmatrix}
	=
	\begin{pmatrix}
		A_N & B_N\\
		C_N & D_N
	\end{pmatrix}
	\begin{pmatrix}
		F(\xi/2^{N})\\
		F(\xi/2^{N+1})
	\end{pmatrix},
	\label{eq:exact}
\end{equation*}
Consequently, the true Riccati function is
\begin{equation*}
	q_{\mathrm{true}}(\xi)=\frac{F(\xi)}{F(\xi/2)}
	=
	\frac{A_N F_N + B_N F_{N+1}}
	{C_N F_N + D_N F_{N+1}},
\end{equation*}
where $F_N:=F(\xi/2^{N})$ and $F_{N+1}:=F(\xi/2^{N+1})$.

The numerical approximation amounts to replaces $F_N$ and $F_{N+1}$ by $1$, i.e.,
\begin{equation*}
	q_{\mathrm{approx}}(\xi)=\frac{A_N+B_N}{C_N+D_N}.
\end{equation*}

Now the truncation error is easily calculated as
\begin{equation*}
	q_{\mathrm{true}}(\xi)-q_{\mathrm{approx}}(\xi)
	=
	\frac{\left|\begin{matrix}
			A_N & B_N\\
			C_N & D_N
		\end{matrix}\right|(F_N-F_{N+1})}
	{F(\xi/2)\,(C_N+D_N)}.
	\label{eq:trunc_exact}
\end{equation*}

Since $\det(M(x))=-p^{2}e^{-\pi i x}$, we have
\begin{equation*}
	\left|\begin{matrix}
		A_N & B_N\\
		C_N & D_N
	\end{matrix}\right| = \prod_{k=0}^{N-1} |\det(M(\xi/2^{k}))|
	= p^{2N}.
\end{equation*}
Moreover, $F$ is a Fourier transform of a probability measure, hence $|F(x)|\le1$ for all real $x$, so $|F_N-F_{N+1}|\le2$. Therefore
\begin{equation*}
	|q_{\mathrm{true}}(\xi)-q_{\mathrm{approx}}(\xi)|
	\le
	\frac{2 p^{2N}}{|F(\xi/2)|\,|C_N+D_N|}.
	\label{eq:trunc_bound}
\end{equation*}

We note that $|C_N+D_N-F(\frac{\xi}{2})|=|(1-F_N)C_N+(1-F_{N+1})D_N|\leq c\frac{|\xi|}{2^N} $, and the observed value of $|F|$ is always larger than $10^{-10}$ hence $C_N+D_N$ can be replaced by $F(\xi/2)$ without affecting the error estimates. However, for $N=80$, $p^{2N}=p^{160}\approx 1.15\times10^{-34}$. Even if $|F(\xi/2)|$ is as small as $10^{-9}$, the truncation error is bounded by
\[
\frac{2\times1.15\times10^{-34}}{10^{-18}} \approx 2.3\times10^{-16},
\]
which is already below double-precision machine epsilon. From the graph \ref{fig:Flog}, it is clear that most of $|F|$ lies above $\frac{1}{|\xi|}$ hence we are relatively safe for  $\xi\in [0, 10^9]$ and in $[0,10^7]$ all the $|F|$ scanned are above $10^{-11}$. Hence the truncation error is negligible in IEEE double precision.

{\it Round-off error propagation.} Note the $\infty$ norm of the matrix
\[\left\|\begin{matrix}
	p & p^{2}e^{-\pi i x}\\
	1 & 0
\end{matrix}\right\|_{\infty}=1,\]
then rounding error will not amplify during matrix multiplication. Hence the total error for $F(\xi)$ is estimated to be around $N$ times machine epsilon, which is around $10^{-14}$. Hence is well below the observed minimum of $|F|$. Now noting $q(\xi)=\frac{F(\xi)}{F(\xi/2)}$ the error for $p$ is bounded by
\[\frac{|q|\Delta F(\xi)}{|F(\xi)|}+\frac{|q|\Delta F(\xi/2)}{|F(\xi/2)|},\] since the obeserved minimum of $|F|$ for $\xi\in [0,10^7]$ is around $10^{-11}$, this value will be around $10^{-2}|q|$ which is negeligible for $q$.

In practice, the main source of uncertainty is not the pointwise evaluation error, but the finite sampling step used in the coarse scans. That issue is addressed below:

\subsubsection{Error analysis for finite sampling step}

Let \(x_0<x_1<x_2\) be three consecutive grid points with \(x_2-x_0=2h\), and let \(P(\xi)\) be the unique quadratic polynomial interpolating \(q\) at these points. The standard interpolation error formula for a complex-analytic function on the real axis gives (for real and imaginary parts, respectively, note $\eta$ might differ)
\[
\text{Re(Im)}\left(q(\xi)-P(\xi)\right)=\frac{(\text{Re(Im)}q)'''(\eta)}{6}(\xi-x_0)(\xi-x_1)(\xi-x_2),\quad \eta\in[x_0,x_2],
\]
whenever \(q'''\) exists. Hence
\[
|q(\xi)-P(\xi)| \le \frac{M_3}{3\sqrt{2}}\,|\xi-x_0|\,|\xi-x_1|\,|\xi-x_2|,
\]
where \(M_3=\max_{x_0\le x\le x_2}|q'''(x)|\). Taking maximum for the right hand side, we then obtain
\[
|q(\xi)-P(\xi)| \le \frac{2}{9\sqrt{6}}\,M_3 h^3,\qquad \forall \xi\in[x_0,x_2].
\]

Then, if \(\xi_0\) is a zero of \(q\), 
\[
\min_{\xi\in[x_0,x_2]} |P(\xi)| \le |P(\xi_0)|\leq \frac{2}{9\sqrt{6}}M_3 h^3.
\]
Thus, for a fixed tolerance \(\tau>0\), as long as step size \(h< \left(\frac{9\sqrt{6}\tau}{2M_3}\right)^{\frac{1}{3}}\), the above fast valley screening method will capture any potential zeroes that lie within the interval $[x_0,x_2]$, if we don't take into account of the computation errors (which by the above analysis is negligible). The constant $M_3$ itself is bounded by certain weighted average modulus of $p$ in a neighborhood of $[x_0, x_2]$ of the complex plane (by the Cauchy integral formula). In our numerical experiments, the bound \(M_3\) is found to be stable and mostly $<300$; With \(h=0.0625\) and $\tau=0.01$, the required inequality for $h$ is then satisfied, and we will use these values in our experiment.

Given the above error estimates, the combination of a uniform grid with step \(h=0.0625\), complex quadratic interpolation, and fast golden-section screening is able to capture all local minima of \(|q|\) that have depth below \(\tau=10^{-2}\) when the bounds ($|F|$, $|q'''|$ etc.) in the error analysis are stable. The procedure is then complete in the sense that no zero can escape detection. The final refinement on the original function then provides an accurate value of \(|q|\) at the minimum, allowing us to decide whether the minimum is consistent with a true zero or merely a shallow valley.

\subsection{Main numerical results}

As the interval $[0,1]$ is already covered in thereotical analysis, we applied the algorithm to the interval \([1,10^7]\) with step size \(h=0.0625\) and screening threshold \(\tau=10^{-2}\).

\paragraph{Observed minima of \(|q|\).}
No point with \(|q(\xi)|<10^{-8}\) was found. The smallest values of \(|q|\) encountered in the entire scan are of order \(10^{-6}\). Table~\ref{tab:minima} lists the ten deepest minima detected.

\begin{table}[h]
	\centering
	\caption{The ten smallest values of \(|q(\xi)|\) found in the interval \([1,10^7]\).}
	\label{tab:minima}
	\begin{tabular}{r|c|l}
		\hline
		\(\xi\) & \(|q(\xi)|\) & $|F|$\\
		\hline
		6580077.19160276  & $4.06527633980225\times 10^{-7}$ & $7.43789361508972\times 10^{-11}$ \\
		5065070.60771886  & $4.43364745955291\times 10^{-7}$ & $3.32971488894357\times 10^{-11}$ \\
		9084534.95396908  & $8.06008614877524\times 10^{-7}$ & $2.62843487266365\times 10^{-10}$ \\
		8643446.91961220  & $8.18114308901849\times 10^{-7}$ & $1.18414733561572\times 10^{-9}$  \\
		1810801.19998039  & $8.25295565910330\times 10^{-7}$ & $7.16656736452115\times 10^{-10}$ \\
		2717521.23238440  & $8.88889998900817\times 10^{-7}$ & $2.45081991050125\times 10^{-10}$ \\
		4306582.88129792  & $1.36750074319750\times 10^{-6}$ & $6.47484921635307\times 10^{-10}$ \\
		883528.97039339   & $1.71091766017875\times 10^{-6}$ & $1.24941266331721\times 10^{-9}$  \\
		6521201.19741959  & $1.79104246408162\times 10^{-6}$ & $2.13335485173019\times 10^{-9}$  \\
		4019830.95206717  & $2.02649732849255\times 10^{-6}$ & $7.32702878565715\times 10^{-10}$ \\
		\hline
	\end{tabular}
\end{table}

\paragraph{Distribution of \(|q|\).}
Figure~\ref{fig:q_scan} and Figure~\ref{fig:q_scan2} display the graph of \(|q(\xi)|\) over two representative subinterval (e.g., around $1427800$ and $10025$, one coarse and one fine), showing the typical oscillatory behaviour with period \(2\) and aggregation of values in the $[0,1]$ interval. Over the whole scanned range, the oscillatory profile and value aggregations look similar and the minimum value never drops below \(1\times10^{-7}\). Notice the red stars indicate the local minimum points that pass the fast valley screening.

\begin{figure}[h]
	\centering
	\includegraphics[width=0.8\textwidth]{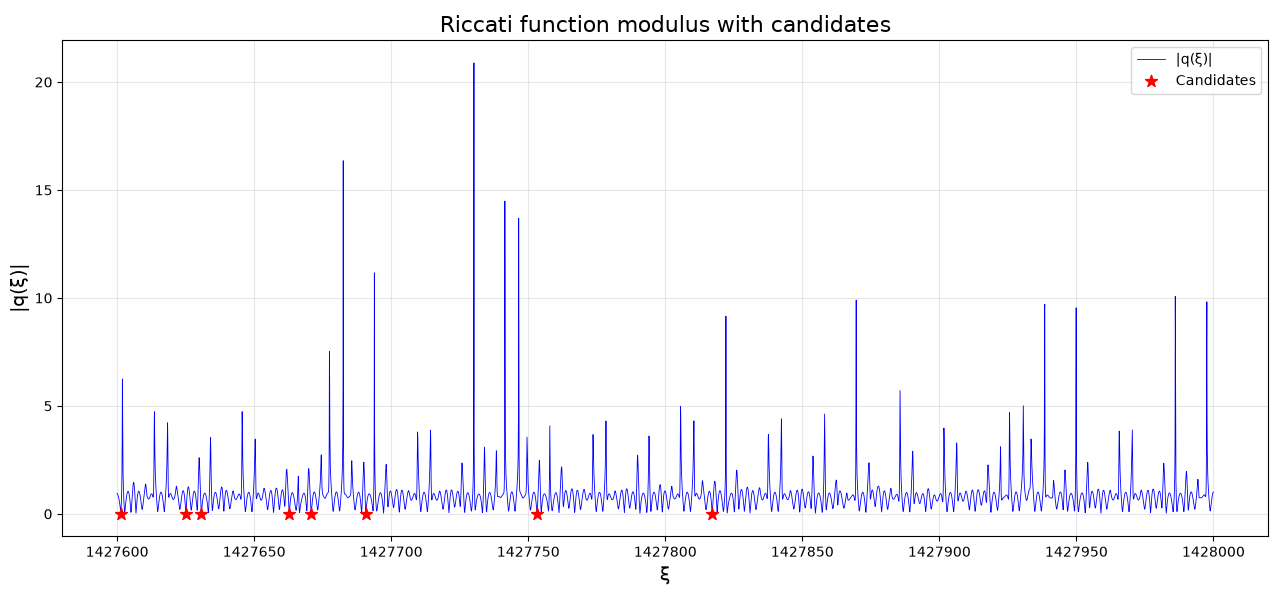}
	\caption{Plot of \(|q(\xi)|\) including interval \([1427600,1428000]\).}
	\label{fig:q_scan}
\end{figure}

\begin{figure}[h]
	\centering
	\includegraphics[width=0.8\textwidth]{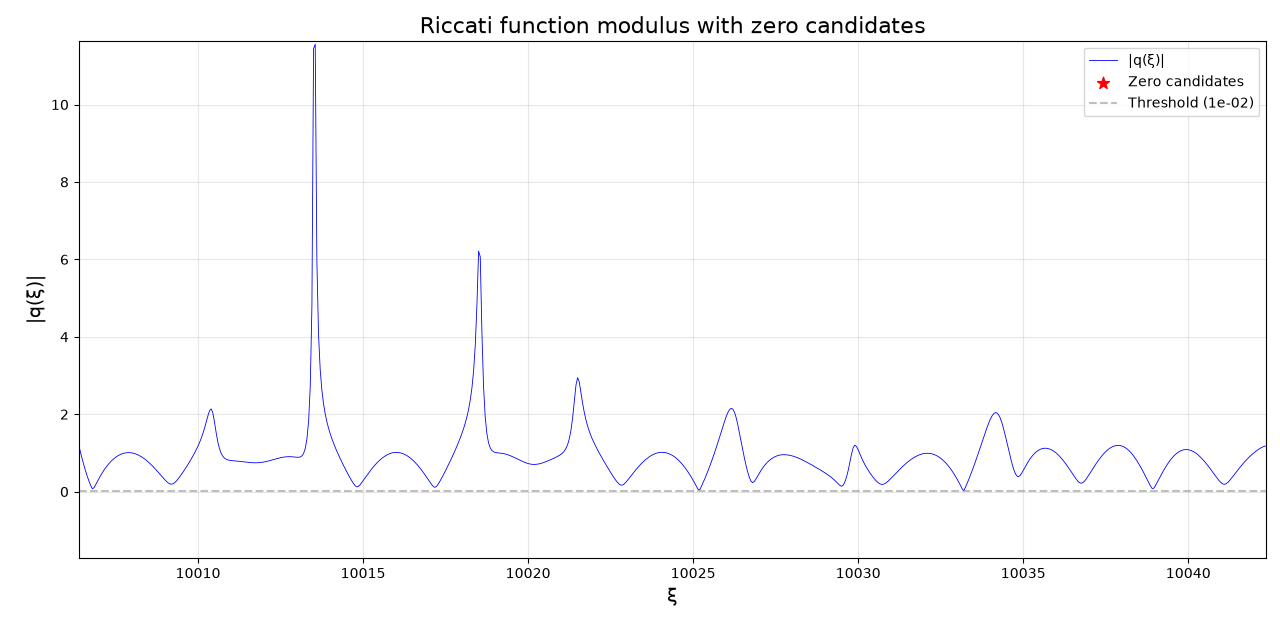}
	\caption{Plot of \(|q(\xi)|\) including interval \([10010, 10040]\). No zero candidates in this interval}
	\label{fig:q_scan2}
\end{figure}

Figure~\ref{fig:q_traj} shows the trajectory of $q$ on the complex plane with $\xi$ around $10000$. Notice the period 2 oscillatory behavior.

\begin{figure}[h]
	\centering
	\includegraphics[width=0.8\textwidth]{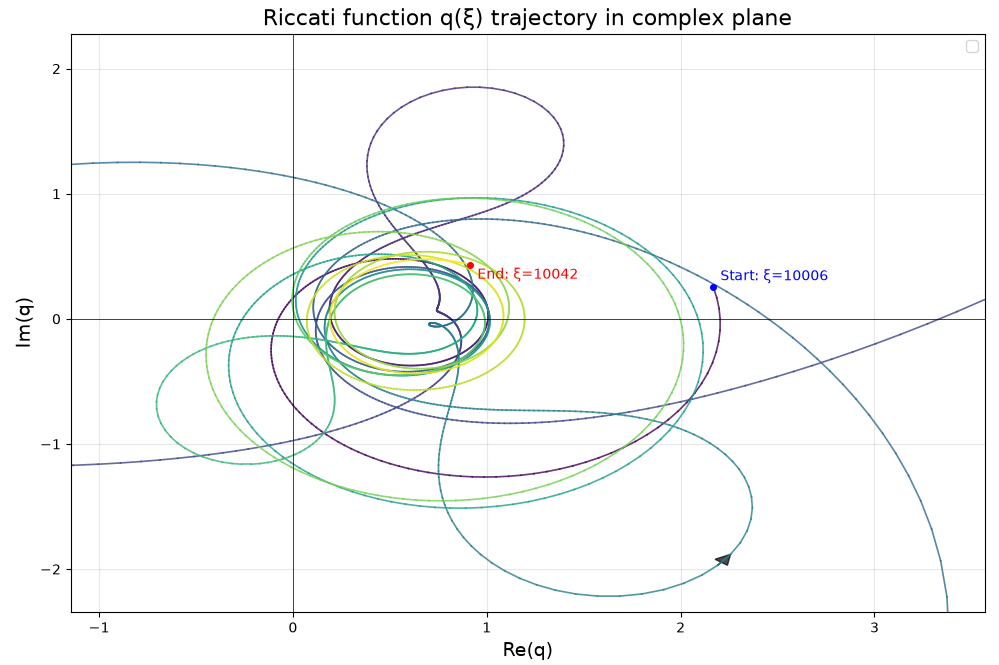}
	\caption{Plot of trajectories of \(q(\xi)\) with $\xi\in [10006, 10042]$. The color gets lighter when $\xi$ increases.}
	\label{fig:q_traj}
\end{figure}

The absence of any real zero of \(q\) (and hence of \(F\)) in the scanned range, strongly indicates that \(F\) has no real zeros at all.

\medskip

\noindent\textit{Remark.} The present work is numerical and does not constitute a formal proof. A rigorous verification would require a computer-assisted proof, which is left for future investigation. Nevertheless, consistency of our results across multiple independent checks (grid refinements, different thresholds, and comparison with direct computation through the Riccati equation) and above error analysis lends strong support to the conclusion that \(F\) has no real zeros.

\section*{Acknowledgements}
The authors would like to thank Professors Xing-Gang He and Xin-Rong Dai for drawing our attention to these problems. The second author would like to thank Professors Meng Wu and Yuan-Yang Chang for helpful discussions on the Fourier decay of self-similar measures; thank Professor Guo-Tai Deng for helpful discussion on the convolution structure of non-homogeneous self-similar measure. The authors were supported by the National Natural Science Foundation of China under Grants 12301131, 12271534 and 12301105.



\begin{thebibliography}{99}
\bibitem{AFL19}
{\sc L. X. An, X. Y. Fu, C. K. Lai}, {\it On spectral Cantor-Moran measures and a variant of Bourgain's sum of sine problem},  Adv. Math. \textbf{349} (2019), 84--124.
\bibitem{Brualdi09}
R.~Brualdi, {\it Introductory Combinatorics}, 4th ed.,
Pearson/Prentice Hall, Upper Saddle River, 2004.
\bibitem {Dai12}
{\sc X.-R. Dai}, {\it When does a Bernoulli convolution admit a
spectrum?},  Adv. Math. \textbf{231} (2012), 1681--1693.
\bibitem{Deng26}
{\sc G.-T. Deng}, {\it Private communication}.
\bibitem {DC21}
{\sc Q.-R. Deng, J. B. Chen}, {\it Uniformlity of spectral self-affine measures},  Adv. Math. \textbf{380} (2021), 107568.
\bibitem {DHLai19}
{\sc D. Dutkay, J. Haussermann, C.-K. Lai}, {\it Hadamard triples generate self-affine spectral measures}, Trans. Amer. Math. Soc. \textbf{371} (2019), 1439-1481.

\bibitem {DL14}
{\sc D. Dutkay, C.-K. Lai}, {\it Uniformly of measures with Fourier frames},  Adv. Math. \textbf{252} (2014), 684--707.


\bibitem{Fug74}
{\sc B. Fuglede}, {\it Commuting self-adjoint partial differential operators and a group theoretic problem},  J. Funct. Anal. \textbf{16} (1974), 101--121.

\bibitem{HKTW18}
{\sc X. G. He, Q. C. Kang, M. W. Tang, Z. Y. Wu}, {\it Beurling dimension and self-similar measures,} J. Funct. Anal. \textbf{274} (2018), 2245--2264.

\bibitem {Hut81}
{\sc J. Hutchinson}, {\it Fractals and self-similarity}, Indiana Univ. Math. J., \textbf{30} (1981), 713--747.

\bibitem  {JP98}
{\sc P. Jorgensen, S. Pedersen}, {\it Dense analytic subspaces in
fractal $L^2$-spaces}, J. Anal. Math. \textbf{75} (1998), 185--228.

\bibitem  {LW02}
{\sc I. {\L}aba, Y. Wang}, {\it On spectral Cantor measures}, J.
Funct. Anal.  \textbf{193} (2002), 409--420.

\bibitem{LS22}
{\sc J. L. Li, T. Sahlsten}, {\it Trigonometric series and self-similar sets}, J. Eur. Math. Soc. \textbf{24} (2022), 341--368.

\bibitem{Z26}
{\sc T. Zhang}, {\it Both directions of Fuglede's conjecture fail in dimension two}, https://arxiv.org/pdf/2607.15632.

\end{thebibliography}
\end{document}